\documentclass[11pt,leqno]{article}

\usepackage{amssymb,amsmath,amsthm,mysects,url,rotating}
 \usepackage{graphicx,epsfig}
 \usepackage{xcolor,graphicx}
\newcommand{\n}{\noindent}

\newcommand{\vp}{\varepsilon}
\newcommand{\bb}[1]{\mathbb{#1}}
\newcommand{\cl}[1]{\mathcal{#1}}

\theoremstyle{plain}
\newtheorem{thm}{Theorem}[section]
\newtheorem{lem}[thm]{Lemma}

\theoremstyle{definition}

\theoremstyle{remark}
\newtheorem{rem}[thm]{Remark}

\numberwithin{equation}{section}

\def\CC{\bb C}

\def\E{\bb E}

\def\P{\bb P}

\def\CC{\bb C}

\begin{document}

\title{Quantum Expanders and Geometry of Operator Spaces II}

\author{by\\
Gilles  Pisier\\
Texas A\&M University\\
College Station, TX 77843, U. S. A.\\
and\\
Universit\'e Paris VI\\
IMJ, Equipe d'Analyse Fonctionnelle, Case 186,\\ 75252
Paris Cedex 05, France}

 \maketitle

         \bigskip

         \def\o{\overline}
            \def\tr{{\rm tr }}
        In this appendix to \cite{Pq} we give a quick proof of an inequality that can be substituted to Hastings's
        result from \cite{Ha}, quoted  as  Lemma 1.9 in  \cite{Pq}. Our inequality is less sharp but also
        appears to apply with more general (and even matricial) coefficients. It shows that up to a universal constant
        all moments of the norm of a linear combination of  the form $$S=\sum\nolimits_j   a_j  U_j \otimes \bar U_j (1-P)$$ are dominated
        by those of  the corresponding Gaussian sum
        $$S'=\sum\nolimits_j  a_j Y_j \otimes \bar Y'_j .$$
        The advantage is that  $S'$ is now simply separately  a Gaussian random variable
        with respect to the independent Gaussian random  matrices $(Y_j)$ and $(Y'_j)$.\\
        We recall that we denote by $P$
        the orthogonal projection onto the orthogonal of the identity. 
        Also recall we denote by $S_2^N$ the space $M_N$ equipped with the Hilbert-Schmidt norm ($S_2^N$
        can also be naturally  identified with $\ell_2^N \otimes_2 \o{\ell_2^N}$).
         We will view elements
        of the form $\sum x_j \otimes \bar y_j$ with $x_j,y_j\in M_N$
        as linear operators acting on $S_2^N$  as follows
        $$T(\xi)= \sum\nolimits_j x_j\xi y_j^*,$$
        so that      \begin{equation}\label{i34}\|\sum x_j \otimes \bar y_j\|=\|T\|_{B(S_2^N)}.\end{equation}
        
        We denote by $(U_j)$   a sequence of i.i.d. random $N\times N$-matrices uniformly distributed over the unitary group $U(N)$.
        We will denote by  $(Y_j)$ a sequence of i.i.d. Gaussian random $N\times N$-matrices, more precisely each $Y_j$ is distributed like the variable $Y$
        that is such that  $\{Y(i,j) N^{1/2}\}$ is a standard family of $N^2$ independent complex Gaussian variables with mean zero and variance 1.
        In other words $Y(i,j) =(2N)^{-1/2}(g_{ij} + \sqrt{-1} g'_{ij})$ where  $g_{ij} ,g'_{ij} $ are independent Gaussian normal $N(0,1)$ random variables.
        
        We denote by $(Y'_j)$ an independent copy of $(Y_j)$.
        
        We will denote by $\|.\|_q$ the Schatten $q$-norm ($1\le q\le \infty$), i.e.
        $\|x\|_q=(\tr(|x|^q))^{1/q}$, with the usual convention that
        for $q=\infty$ this is the operator norm.
        \begin{lem} There is an absolute constant $C$
        such that for any $p\ge 1$ we have for any scalar sequence $(a_j)$ and any $1\le q\le \infty$
        $$\E\|\sum_1^n  a_j U_j \otimes \bar U_j (1-P)\|_q^p \le C^p \E \|\sum_1^n a_j Y_j \otimes \bar Y'_j \|_q^p,$$
        (in fact this holds for all $k$ and all matrices $a_j\in M_k$
        with $   a_j\otimes $ in place of $a_j$).
        \end{lem}
           \begin{proof} We assume that all three sequences $(U_j)$, $(Y_j)$ and $(Y'_j)$ are mutually independent.
           The proof is based on the well known fact that the sequence $(Y_j)$ has the same distribution
           as $U_j|Y_j|$, or equivalently that the two factors
           in the polar decomposition $Y_j=U_j|Y_j|$ of $Y_j$ are mutually independent. Let $\cl E$  denote the conditional expectation operator
           with respect to the $\sigma$-algebra generated by $(U_j)$. Then we have
           $U_j \E|Y_j|= {\cl E}( U_j|Y_j|)={\cl E}( Y_j)$, and moreover
           $$(U_j\otimes \bar U_j)  \E(|Y_j| \otimes \o{|Y_j|})={\cl E}( U_j|Y_j| \otimes \o{U_j|Y_j|})={\cl E}(Y_j \otimes \o{Y_j}) .$$
           Let
           $$T=\E(|Y_j| \otimes \o{|Y_j|})=\E(|Y| \otimes \o{|Y|}).$$
           Then we have
           $$ \sum a_j (U_j\otimes \bar U_j) T (I-P) ={\cl E}((  \sum a_j Y_j \otimes \o{Y_j})(I-P)).$$
           Note that by rotational invariance of the Gaussian measure
           we have
           $(U \otimes \bar U)T(U^* \otimes \bar U^*)=  T$. Indeed   since
           $UYU^*$ and $Y$ have the same distribution
           it follows that
           also $UYU^* \otimes \o{UYU^*} $ and $Y\otimes \bar Y$  have the same distribution,
           and hence so do their modulus.\\
           Viewing $T$ as a linear map on $S_2^N=\ell_2^N \otimes \o{\ell_2^N} $,
           this yields
           $$\forall U\in U(N)\quad T (U\xi U^*)= UT (\xi )U^*.$$
          Representation theory 
           shows that $T$ must be simply    a linear combination of $P$ and $I-P$.
           Indeed, the unitary representation $U\mapsto U\otimes \bar U$ 
           on $U(N)$ decomposes
           into exactly two distinct irreducibles, by restricting either
           to the subspace  $\CC I$ or its orthogonal. Thus, by Schur's Lemma we know a priori
           that there are two scalars $\chi'_N ,\chi_N$ such that $T=\chi'_N P+ \chi_N (I-P)$.
           We may also observe $\E( |Y|^2)=I$ so  that $T(I)=I$ and hence $\chi'_N=1$, therefore
           $$T=P+\chi_N (I-P).$$
           Moreover, since $T(I)=I$ and $T$ is self-adjoint, $T$ commutes with $P$ and  hence $T(I-P)=(I-P)T$, 
            so that we have
           \begin{equation}\label{37}  \sum_1^n  a_j (U_j \otimes \bar U_j )(1-P) T = {\cl E} \sum_1^n a_j (Y_j \otimes \bar Y_j )(I-P).\end{equation}
           We claim that $T$ is invertible and that
           there is an absolute constant $C$ so that
           $$\|T^{-1}\| ={\chi_N}^{-1}\le C.$$
           From this and \eqref{37}  follows immediately that for any $p\ge 1$
          \begin{equation}\label{36} \E\|\sum_1^n  a_j (U_j \otimes \bar U_j) (1-P)\|_q^p \le C^p  \E\| \sum_1^n a_j (Y_j \otimes \bar Y_j) (1-P)  \|_q^p.\end{equation}
        
          To check the claim it suffices to compute $\chi_N$. For $i\not=j$ we have a priori
            $T (e_{ij})= e_{ij} \langle T(e_{ij}), e_{ij}\rangle$ but (since $\tr (e_{ij}  )=0$) we know $T (e_{ij})=\chi_N e_{ij}$.
            Therefore for any $i\not=j$ we have $\chi_N = \langle T(e_{ij}), e_{ij}\rangle$,
                       and the latter we can compute
                     $$\langle T(e_{ij}), e_{ij}\rangle =\E\tr (|Y|e_{ij} |Y|^*e^*_{ij})= \E( |Y|_{ii} |Y|_{jj}).$$
         Therefore, 
         $$N(N-1)  \chi_N= \sum_{i\not=j}  \E( |Y|_{ii} |Y|_{jj})=\sum_{i,j}  \E( |Y|_{ii} |Y|_{jj})-\sum_j  \E(  |Y|^2_{jj})=\E(\tr|Y |)^2 -N \E(  |Y|^2_{11}).$$
         Note that $\E(  |Y|^2_{11})= \E  \langle |Y| e_1,e_1\rangle ^2 \le   \E \langle |Y|^2 e_1,e_1\rangle =
         \E\|Y (e_1)\|^2_2=1$, and hence
         $$N(N-1)  \chi_N= \sum_{i\not=j}  \E( |Y|_{ii} |Y|_{jj})\ge \E(\tr|Y |)^2 -N  .$$
         Now it is well known that $E|Y| =b_N I$ where $b_N$ is determined by
         $b_N=N^{-1}\E \tr |Y|= N^{-1}\|Y\|_1$ and $\inf_N b_N >0$ (see e.g. \cite[p.80]{MP}). Actually, by Wigner's limit theorem,
         when $N \to \infty$, $N^{-1}\|Y\|_1$ tends almost surely to $\tau|c_1|$.
         Therefore, $ N^{-2}\E(\tr|Y |)^2 $ tends to $(\tau|c_1|)^2$.  
         We have
         $$  \chi_N=(N(N-1))^{-1} \sum_{i\not=j}  \E( |Y|_{ii} |Y|_{jj})\ge (N(N-1))^{-1} \E(\tr|Y |)^2 -(N-1)^{-1}  ,$$
         and this implies
           $$\liminf_{N\to \infty}   \chi_N \ge (\tau|c_1|)^2.$$
 In any case, we have
         $$\inf_N \chi_N >0,$$
         proving our claim.

                    We will now deduce from \eqref{36} the desired estimate by a classical decoupling
                    argument for multilinear expressions in Gaussian variables.\\
                    We first observe $\E ((Y \otimes \bar Y) (I-P))=0$. Indeed, 
                    by orthogonality, a simple calculation shows that  $\E (Y \otimes \bar Y)=\sum _{ij} \E (Y_{ij}\o{Y_{ij}})e_{ij}\otimes \o{e_{ij}}= \sum _{ij} N^{-1}e_{ij}\otimes \o{e_{ij}}=P $,
                    and hence $\E ((Y \otimes \bar Y) (I-P))=0$.
                    
                    We will use 
                    $$(Y_j, Y'_j ) {\buildrel {dist} \over {=}}( (Y_j+Y'_j )/\sqrt 2,  (Y_j-Y'_j )/\sqrt 2)$$
                    and
                    if $\E_Y$ denotes the conditional expectation with respect to
                    $Y$ we have (recall $\E (Y_j \otimes \bar Y_j )(I-P)=0$)
                    $$\sum_1^n a_j Y_j \otimes \bar Y_j (I-P)= \E_Y (\sum_1^n a_j Y_j \otimes \bar Y_j (I-P)-
                    \sum_1^n a_j Y'_j \otimes \bar Y'_j(I-P)).$$
                  Therefore    
                   $$ \E\|\sum_1^n  a_j Y_j \otimes \bar Y_j (1-P)\|_q^p \le
                   \E\|\sum_1^n  a_j Y_j \otimes \bar Y_j (1-P)- \sum_1^n a_j Y'_j \otimes \bar Y'_j(I-P))\|_q^p   $$
                    
                   $$=  \E\|\sum_1^n  a_j (Y_j+Y'_j )/\sqrt 2 \otimes \o{ (Y_j+Y'_j )/\sqrt 2} (1-P)- \sum_1^n a_j (Y_j-Y'_j )/\sqrt 2 \otimes \o{ (Y_j-Y'_j )/\sqrt 2}(I-P))\|_q^p $$                   
     $$= \E\|\sum_1^n  a_j  (Y_j    \otimes \o{Y'_j }+Y'_j    \otimes \o{Y_j }  ) (1-P) \|_q^p
           $$
           and hence by the triangle inequality
           $$\le 2^p \E\|\sum_1^n  a_j  (Y_j    \otimes \o{Y'_j }   ) (1-P) \|_q^p.$$
Thus we conclude a fortiori
$$\E\|\sum_1^n  a_j U_j \otimes \bar U_j (1-P)\|_q^p \le (2C)^p \E\|\sum_1^n  a_j  (Y_j    \otimes \o{Y'_j }   )  \|_q^p.$$
 \end{proof}
    \begin{thm} Let $C$ be as in the preceding Lemma. Let 
$$\hat S^{(N)}= \sum_1^n  a_j U_j \otimes \bar U_j (1-P).$$
Then 
 \begin{equation}\label{a4}\limsup_{N\to \infty}\E \|\hat S^{(N)} \| \le   4C ( \sum |a_j|^2)^{1/2}. \end{equation}
Moreover we have 
almost surely
 \begin{equation}\label{a5} \limsup_{N\to \infty} \|\hat S^{(N)} \| \le   4C ( \sum |a_j|^2)^{1/2}. \end{equation}
     \end{thm}
   \begin{proof} A very direct argument is indicated in Remark \ref{sim2} below, but we prefer
   to base the proof  on \cite{HT2} in the style of    \cite{Pq} in order to make clear that it remains valid
   with matrix coefficients.
By    \cite[(3.1)]{Pq} applied twice (for $k=1$) (see also   Remark 3.5 in \cite{Pq})
one finds for 
any even integer $p$
 \begin{equation}\label{a1}
E\tr |\sum_1^n  a_j  (Y_j    \otimes \o{Y'_j }   ) |^p  \le (\E \tr|Y|^p)^2 ( \sum |a_j|^2)^{p/2}
 \end{equation}
Therefore by the preceding Lemma
 $$E\tr |\hat S^{(N)} |^p  \le C^p (\E \tr|Y|^p)^2 ( \sum |a_j|^2)^{p/2}  ,$$
and hence a fortiori
$$E\|\hat S^{(N)} \|^p\le  N^2 C^p (\E \|Y\|^p)^2 ( \sum |a_j|^2)^{p/2} .$$
We then complete the proof, as in \cite{Pq}, using only the concentration
of the variable $ \|Y\|$. We have an absolute constant $\beta'$  and   $\vp(N)>0$ 
tending to zero when $N\to \infty$, such that
$$(\E \|Y\|^p)^{1/p}  \le 2+\vp(N) + \beta' \sqrt{p/N},  $$
and hence
$$(E\|\hat S^{(N)} \|^p) ^{1/p} \le N^{2/p}C (2+\vp(N) + \beta' \sqrt{p/N})^2 ( \sum |a_j|^2)^{1/2} .
$$
Fix $\vp>0$ and  choose $p$ so that $N^{2/p} = \exp \vp $, i.e. $p=  2\vp^{-1} \log N$
(note that this is $\ge 2$ when $N$ is large enough)
  we obtain
  $$E\|\hat S^{(N)} \|  \le (E\|\hat S^{(N)} \|^p) ^{1/p} \le 4 e^\vp C(1+\vp^{-1} \vp'(N) ) ( \sum |a_j|^2)^{1/2}$$
  where $\vp'(N)\to 0$ when $N\to \infty$, and \eqref{a4} follows.\\ Let $R_N=4C(1+\vp^{-1} \vp'(N) ) ( \sum |a_j|^2)^{1/2}$.
 By Tshebyshev's inequality $(E\|\hat S^{(N)} \|^p) ^{1/p} \le e^\vp R_N$ implies
  $$\P\{ \|\hat S^{(N)} \| > e^{2\vp} R_N\}\le     \exp -\vp p=N^2.$$
  From this it is immediate  that almost surely
 $$ \limsup_{N\to \infty} \|\hat S^{(N)} \| \le e^{2\vp} 4 C( \sum |a_j|^2)^{1/2}$$
 and hence \eqref{a5} follows.
    \end{proof}
   
    \begin{rem}\label{sim1}  The same argument can be applied when $a_j\in M_k$ for any integer $k>1$. 
    Then we find
    $$\limsup_{N\to \infty}\E \| \sum_1^n  a_j \otimes U_j \otimes \bar U_j (1-P) \| \le   4C
  \max\{ \|\sum a_j^*a_j\|^{1/2}, \|\sum a_ja_j^*\|^{1/2}   \} .$$
Moreover we have 
almost surely
 $$\limsup_{N\to \infty} \| \sum_1^n  a_j \otimes U_j \otimes \bar U_j (1-P) \| \le   4C
  \max\{ \|\sum a_j^*a_j\|^{1/2}, \|\sum a_ja_j^*\|^{1/2}   \} .$$   
    \end{rem}
        \begin{rem}\label{sim2} In the case of scalar coefficients $a_j$
        the proof extends also to double sums
        of the form
        $$\sum _{ij} a_{ij} U_i \otimes \bar U_j (I-P).$$
  We refer the reader to  \cite[Theorem 16.6]{Pg} for a self-contained proof of \eqref{a1}
  for such double sums.
        
 \end{rem}
    \n\textbf{Acknowledgment.}  I am   grateful to S. Szarek for a useful suggestion.

  \end{document}